\documentclass[a4paper,11pt]{article}
\usepackage[utf8]{inputenc}

\usepackage{amsthm,amsmath,amsfonts,amssymb}
\usepackage[english]{babel}
\usepackage{graphicx}
\usepackage{verbatim}
\usepackage{color}

\DeclareMathOperator{\diag}{diag}

\newcommand\blfootnote[1]{%
  \begingroup
  \renewcommand\thefootnote{}\footnote{#1}%
  \addtocounter{footnote}{-1}%
  \endgroup
}

\newcommand{\mytilde}{\raise.17ex\hbox{$\scriptstyle\mathtt{\sim}$}}
\def\Cay{\mathop{\rm Cay }\nolimits}
\def\Circ{\mathop{\rm Circ }\nolimits}

\def\det{\mathop{\rm det }\nolimits}

\def\diag{\mathop{\rm diag }\nolimits}

\def\mod{\mathop{\rm mod }\nolimits}

\def\Z{\ns{Z}}

\def\Z{\ns Z}

\def\b{\mbox{\boldmath $b$}}

\def\vec0{\mbox{\boldmath $0$}}

\def\G{\Gamma}

\def\S{\mbox{\boldmath $S$}}

\def\Z{\ns{Z}}

\def\S{\mbox{\boldmath $S$}}

\def\G{\Gamma}

\def\Z{\mathbb Z}

\theoremstyle{plain}   

\newtheorem{theorem}{Theorem}[section]
\newtheorem{proposition}[theorem]{Proposition}

\newtheorem{lemma}[theorem]{Lemma}

\def\b{\textcolor{blue}}

\begin{document}

\title{New Moore-like bounds and some optimal families of Cayley Abelian mixed graphs}

\author{C. Dalf\'o$^a$, M. A. Fiol$^b$, N. L\'opez$^c$\\
{\small $^a$Dept. de Matem\`atica, Universitat de Lleida}\\
{\small Igualada (Barcelona), Catalonia}\\
{\small {\tt cristina.dalfo@udl.cat}}\\
{\small $^{b}$Dept. de Matem\`atiques, Universitat Polit\`ecnica de Catalunya} \\
{\small Barcelona Graduate School of Mathematics} \\
{\small Barcelona, Catalonia} \\
{\small {\tt miguel.angel.fiol@upc.edu}} \\
{\small $^c$Dept. de Matem\`atica, Universitat de Lleida}\\
{\small Lleida, Spain}\\
{\small {\tt nacho.lopez@udl.cat}}\\
}

\date{}
\maketitle
\begin{abstract}
Mixed graphs can be seen as digraphs that have both arcs and edges (or digons,
that is, two opposite arcs). In this paper, we consider the case where such
graphs are Cayley graphs of Abelian groups. Such groups can be constructed by using a generalization to $\Z^n$ of the concept of congruence in $\Z$.
Here we use this approach to present some families of mixed graphs, which, for every fixed value of the degree, have an  asymptotically large number of vertices as the diameter increases. In some cases, the results obtained are shown to be optimal.
\end{abstract}

\noindent\emph{Keywords:} Mixed graph, degree/diameter problem, Moore bound, Cayley graph, Abelian group, Congruences in $\Z^n$.

\noindent\emph{Mathematical Subject Classifications:} 05C35, 05C25, 05C12, 90B10.

\blfootnote{
\begin{minipage}[l]{0.3\textwidth} \includegraphics[trim=10cm 6cm 10cm 5cm,clip,scale=0.15]{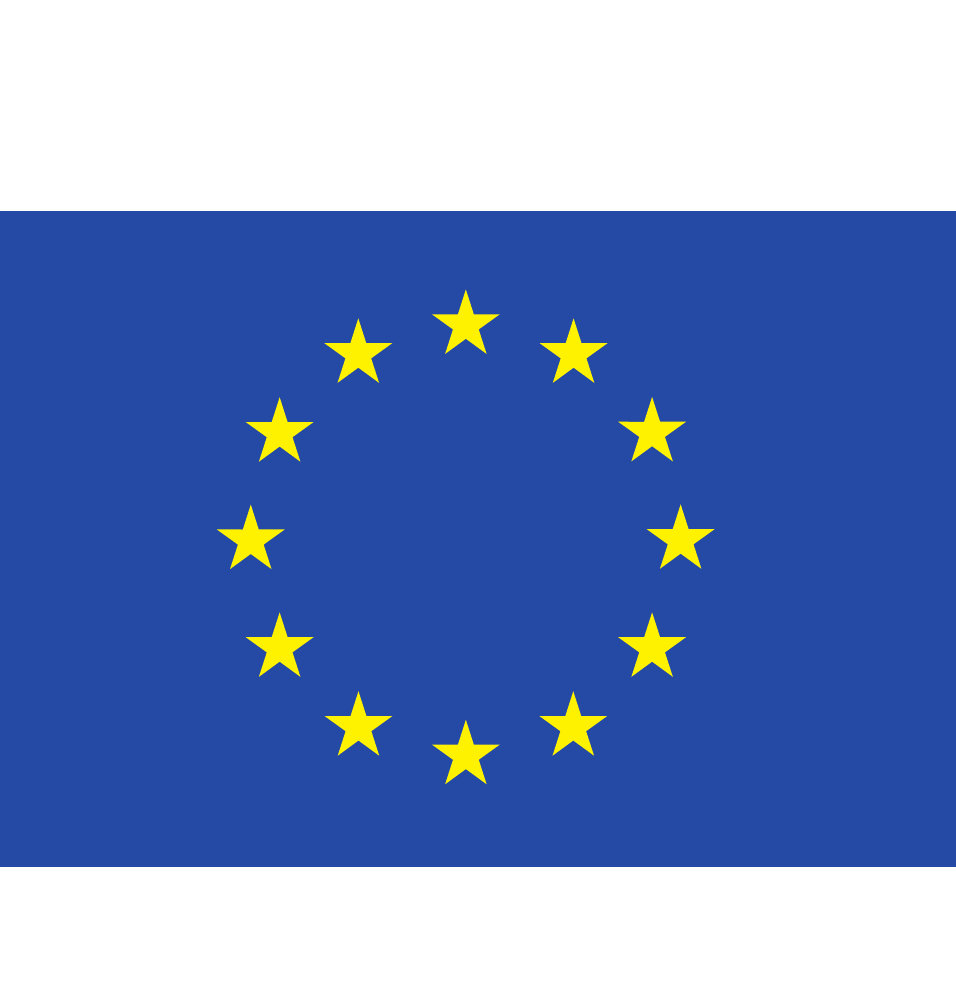} \end{minipage}  \hspace{-2cm} \begin{minipage}[l][1cm]{0.79\textwidth}
   The research of C. Dalf\'o has also received funding from the European Union's Horizon 2020 research and innovation programme under the Marie Sk\l{}odowska-Curie grant agreement No 734922.
  \end{minipage}}

\section{Introduction}

The choice of the interconnection network for a multicomputer
or any complex system  is one of the crucial problems the designer has to face.
In fact, the network topology largely affects the performance of the system and it
has an important contribution to its overall cost. As such topologies are modeled
by either graphs, digraphs, or mixed graphs, this has lead to the following
optimization problems:\\
$(a)$ Find graphs, digraphs or mixed graphs, of given diameter and maximum out-degree that have a large number of vertices.\\
$(b)$ Find graphs, digraphs or mixed graphs, of given number of vertices and maximum out-degree that have small diameter.

For a more detailed description of these problems, their possible applications, the usual notation, and the theoretical background, see the comprehensive survey of Miller and \v{S}ir\'a\v{n} \cite{Moore-survey}. For more specific results concerning mixed graphs, which are the topic of this paper, see, for example, Nguyen and Miller \cite{nm08}, and Nguyen, Miller, and Gimbert \cite{nmg07}.

Mixed graphs, with their undirected and directed connections, are a better representation of a complex network (as, for example, the Internet) than the one given by both graphs or digraphs. As two examples of the applications of mixed graphs, we have:\\
$(a)$ Mixed graphs modeling job shop scheduling problems, in which a collection of tasks is to be performed. In this case, undirected edges represent a constraint for two tasks to be incompatible (they cannot be performed simultaneously). Directed edges may be used to model precedence constraints, in which one task must be performed before another.\\
$(b)$ Mixed graphs are also used as models for Bayesian inference. The directed edges of these graphs are used to indicate a causal connection between two events, in which the outcome of the first event influences the probability of the second event. Undirected edges, instead, indicate a non-causal correlation between two events.

In this first section, we give some preliminaries, together with some details on Abelian Cayley graphs from congruences in $\mathbb{Z}^n$. In Section \ref{sec:new-approach}, we provide a combinatorial way to obtain the Moore bound for mixed Abelian Cayley graphs. Taking into account some symmetries, we rewrite in a more condensed way the Moore bound for some cases. Finally, in Section \ref{sec:requalto1}, we obtain some Moore Abelian Cayley mixed graphs when they exist, and some infinite families of dense graphs when they do not exist.

\subsection{Preliminaries}\label{sec:intro}

The degree/diameter or $(d,k)$ problem asks for constructing the largest possible graph (in terms of the number of vertices), for a given maximum degree and a given diameter. In the degree/diameter problem for mixed graphs we have three parameters: a maximum undirected degree $r$, a maximum directed out-degree $z$, and diameter $k$. A natural upper bound for the maximum number of vertices $M(r,z,k)$ for a graph under such degrees and diameter restrictions is (see Buset, El Amiri, Erskine, Miller, and P\'erez-Ros\'es \cite{mixedmoore}):

\begin{equation}\label{eq:new}
M_{z,r,k} = A \frac{u_1^{k+1}-1}{u_1-1}+B\frac{u_2^{k+1}-1}{u_2-1}
\end{equation}
where, with $d=r+z$ and $v=(d-1)^2+4z$,
\begin{align}
u_1 &=\displaystyle{\frac{d-1-\sqrt{v}}{2}}, \qquad
u_2 =\displaystyle{\frac{d-1+\sqrt{v}}{2}}, \label{u's}\\
A   &=\displaystyle{\frac{\sqrt{v}-(d+1)}{2\sqrt{v}}}, \quad \
B   =\displaystyle{\frac{\sqrt{v}+(d+1)}{2\sqrt{v}}}. \label{A&B}
\end{align}

Besides this general bound given above, researchers are also interested in some particular versions of the problem, namely when the graphs are restricted to a certain class, such as the class of bipartite graphs (which was studied by the authors \cite{DFL18}), planar graphs (see  Fellows, Hell, and Seyffarth \cite{Fell95}, and Tischenko \cite{Ti12}), maximal planar bipartite graphs (see Dalf\'o, Huemer, and Salas \cite{DaHuSa16}), vertex-transitive graphs (see Machbeth, \v{S}iagiov\'a, \v{S}ir\'a\v{n}, and Vetr\'{\i}k \cite{Mac10}, and \v{S}iagiov\'a and Vetr\'{\i}k \cite{Sia07}), Cayley graphs (\cite{Mac10,Sia07} and Vetr\'{\i}k \cite{Ve13}), Cayley graphs of Abelian groups (Dougherty and Faber \cite{Dou04}), or circulant graphs (Wong and Coppersmith \cite{Wong74}, and Monakhova \cite{Mona12}). In this paper, we are concerned with mixed Abelian Cayley graphs.

For most of these graph classes there exist Moore-like upper bounds, which in general are smaller than the Moore bound for general graphs, although some of them are quite close to the Moore bound. For example, the Moore-like upper bound for bipartite mixed graphs is (when $r>0$):
\begin{equation}
\label{Moore-mix-bip}
M_B(r,z,k)=
2\left(A\,\frac{u_1^{k+1}-u_1}{u_1^2-1}+ B\,\frac{u_2^{k+1}-u_2}{u_2^2-1}\right),
\end{equation}
where $u_1$, $u_2$, $A$, and $B$ are given by \eqref{u's} and \eqref{A&B} (see Dalf\'o, Fiol, and L\'opez \cite{DFL18}). The upper bound for mixed Abelian Cayley graphs was given by
L\'opez, P\'erez-Ros\'es, and Pujol\`as in \cite{LOPEZ2016145}: Let $\Gamma$ be an Abelian group, and let $\Sigma$ be a generating set of $\Gamma$ containing $r_1$ involutions and $r_2$ pairs of generators and their inverses, and $z$ additional generators, whose inverses are not in $\Sigma$. Thus, the Cayley graph $\Cay(\Gamma, \Sigma)$ is a mixed graph with undirected degree $r$, where $r = r_1+2r_2$, and directed out-degree $z$.
An upper bound for the number of vertices of $\Cay(\Gamma, \Sigma)$, as a function of the diameter $k$, is
\begin{equation}
\label{eq:upper1}
M_{AC}(r_1,r_2,z,k)=\sum_{i=0}^k {r_2+z+i \choose i}{r_1+r_2 \choose k-i}.
\end{equation}

Some interesting (proper) cases, that are mentioned later, are $r_1=0, r_2=1, z=1$ and  $r_1=1, r_2=0, z=2$, for which \eqref{eq:upper1} gives the same Moore bound $(k+1)^2$.

Circulant graphs are Cayley graphs over $\mathbb{Z}_n$, and they have been studied for the degree/diameter problem for both the directed and the undirected case. As in the general case, the definition of circulant graphs can be extended to allow both edges and arcs. Let $\Sigma$ be a generating set of $\mathbb{Z}_n$ containing $r_1$ involutions and $r_2$ pairs of generators, together with their inverses, and $z$ additional generators, whose inverses are not in $\Sigma$. The {\em (mixed) circulant graph} $\Circ(n;\Sigma)$ has vertex set $V=\mathbb{Z}_n$, and each vertex $i$ is connected to $i+a \pmod n$ vertices, for all $a \in \Sigma$.  Thus, $\Circ(n;\Sigma)$ has undirected degree $r$, where $r = r_1+2r_2$, and directed degree $z$. In fact, $r_1 \leq 1$, since $\mathbb{Z}_n$ has either one involution (for $n$ even), or none (for $n$ odd).

\subsection{Abelian Cayley graphs from congruences in $\mathbb{Z}^n$}

Let $M$ be an $n\times n$ non-singular integral matrix, and $\Z^{n}$  the additive group of $n$-vectors with
integral components. The set $\Z^{n}M$, whose elements are linear combinations (with integral coefficients) of the rows of $M$ is said to be the lattice generated by $M$.
By the Smith normal form theorem, $M$ is equivalent to the diagonal matrix $S(M)=S=\diag(s_{1},\ldots ,s_{n})$, where $s_1,\ldots, s_n$ are the {\it invariant factors} of $M$, which satisfy $s_i|s_{i+1}$ for $i=1,\ldots,n-1$.  That is, there exist unimodular matrices $U$ and $V$, such that $S=UMV$. The canonical form $S$ is unique, but the unimodular matrices $U$ and $V$ certainly not. However, this fact does not affect the results below. For more details, see Newman \cite{n14}.
The concept of congruence in $\Z$ has the following natural
generalization to $\Z^{n}$ (see Fiol \cite{f87}) . Let $u,v\in \Z^{n}$. We say that
{\it $u$  is congruent with $v$ modulo $M$}, denoted by
$u \equiv v\pmod{M}$, if
\begin{equation}
\label{eq2}
u-v \in \Z^{n}M.
\end{equation}
The Abelian quotient group $\Z^{n}/\Z^{n}M$ is referred to the
{\it group of integral vectors modulo $M$}.
In particular,
when $M=\diag (m_{1}, \ldots ,m_{n})$, the group $\Z^{n}/\Z^{n}M$ is the direct product of the cyclic
groups $\Z_{m_{i}}$, for $i=1, \ldots ,n$.
Let us consider again the Smith normal form of $M$, $S=\diag(s_{1},\ldots ,s_{n})=UMV$. Then, (\ref{eq2}) holds if and only if $uV \equiv vV\pmod{S}$ or, equivalently,
\begin{equation}
uV_{i} \equiv vV_{i} \pmod{s_i},\quad
i=1,2, \ldots ,n,                                                                      \label{eq5}
\end{equation}
where $V_i$ denotes the $i$-th column of $V$.
Moreover, if $r$
is the smallest integer such that $s_{n-r}=1$, hence, $s_{1}=s_{2}=
\cdots =s_{n-r} = 1$ (if there is no such a $r$, let $r=n$), then the
first $n-r$ equations in (\ref{eq5}) are irrelevant, and we only need to
consider the other ones. This allows us to write
\begin{equation}
u \equiv v\pmod{M} \quad \Leftrightarrow \quad
uV' \equiv vV'\pmod{S'},                                         \label{eq6}
\end{equation}
where $V'$ stands for the $n\times r$ matrix obtained from $V$ by 
taking off the first $n-r$ columns, and $\S'=\diag(s_{n-r+1},s_{n-r+2},\ldots ,s_{n})$. So, the (linear) mapping $\phi$ from the vectors modulo
$M$ to the vectors modulo $S'$ given by $\phi (u)=uV'$ is a
group isomorphism, and we can write
\begin{equation}
\label{fund-theo}
\Z^{n}/\Z^{n}M \cong \Z^{r}/\Z^{r}S' = \Z_{s_{n-r+1}} \times \cdots \times \Z_{s_{n}}.
\end{equation}
Notice that since for any Abelian group $\Gamma$ there exists an integral matrix  $M\in \Z^{n \ast n}$  such that $\Gamma\cong\Z^{n}/\Z^{n}M$, the equality in \eqref{fund-theo} is just the fundamental theorem of finite Abelian groups (see also Proposition \ref{pro2.1}$(b)$).
The next proposition contains more consequences of the above
results. For instance, $(b)$ follows from the fact that
$s_{1}s_{2} \cdots s_{n} = d_{n} = m$ and $s_{i}|s_{i+1}$,
for $i=1,2, \ldots ,n-1$.
\begin{proposition}
\label{pro2.1}
\begin{itemize}
\item[$(a)$]
The number of equivalence classes modulo $M$ is
$|\Z^{n}/\Z^{n}M|=m=|\det M|$.
\item[$(b)$]
If $p_{1}^{r_{1}}p_{2}^{r_{2}} \cdots p_{t}^{r_{t}}$ is the prime factorization of $m$, then $\Z^{n}/\Z^{n}M \cong \Z^{r}/\Z^{r}S'$ for some
$r\times r$ matrix $S'$, with $r \leq \max \{r_{i} :1 \leq i \leq t \}$.
\item[$(c)$]
The Abelian group of integral vectors modulo $M$ is cyclic
if and only if $d_{n-1}=$1.
\item[$(d)$]
 Let $r$ be the smallest integer such that $s_{n-r}=1$. Then, $r$ is the rank of $\Z^{n}/\Z^{n}M$ and the last $r$ rows of $V^{-1}$ form a basis
of $\Z^{n}/\Z^{n}M$.
\end{itemize}
\end{proposition}

Let $M$ be an $n \times n$ integral matrix as above.
Let $A=\{a_{1},\ldots, a_{d}\} \subseteq \Z^{n}/\Z^{n}M$. The {\it multidimensional
$($d-step$)$ circulant digraph $G(M,A)$} has as vertex-set the integral
vectors modulo $M$, and every vertex $u$ is adjacent to the vertices
$u+A\pmod{M}$. As in the case of circulants, the {\it
multidimensional $($d-step$)$ circulant graph $G(M,A)$} is defined similarly just requiring $A=-A$.
Clearly, a multidimensional circulant (digraph, graph, or mixed graph) is a Cayley graph of the Abelian group $\Gamma=\Z^{n}/\Z^{n}M$.
In our context, if $\G$ is an Abelian group with generating set $\Sigma$ containing $r_1+2r_2+z$ generators (with the same notation as before), then there exists an integer $n\times n$ matrix $M$ with size $n=r_1+r_2+z$ such that
$$
\Cay(\G,\Sigma)\cong \Cay(\Z^{n}/\Z^{n}M,\Sigma'\},
$$
where $\Sigma'=\{e_1,\ldots,e_{r_1},\pm e_{r_1+1},\ldots,\pm e_{r_1+r_2},e_{r_1+r_2+1},\ldots,e_{r_1+r_2+z}\}$, and the $e_i$'s stand for the unitary coordinate vectors. For example,  the two following Cayley mixed graphs
$$
\Cay(\Z_{24},\{\pm 2,3,12\}),\quad \mbox{and}\quad
\Cay(\Z^3/\Z^3 M, \{\pm e_1,e_2,e_3\}) \mbox{\ with\ }
M=\left(
\begin{array}{ccc}
3 & -2 & 0\\
0 & 4 & 1\\
0 & 0 & 2
\end{array}
\right),
$$
are isomorphic since the Smith normal form of $M$ is $S=\diag(1,1,24)$ and
$$
S=UMV=\left(
\begin{array}{ccc}
-1 & 0 & 0\\
-4 & 1 & 0\\
-8 & 2 & -1
\end{array}
\right)\left(
\begin{array}{ccc}
3 & -2 & 0\\
0 & 4 & 1\\
0 & 0 & 2
\end{array}
\right)\left(
\begin{array}{ccc}
-1 & 0 & 2\\
-1 & 0 & 3\\
0 & 1 & -12
\end{array}
\right).
$$
Indeed, $\Z^3/\Z^3 M$ is a cyclic group of order $|\det M|=24$ and, according to \eqref{eq5}, the generators $\pm e_1$, $e_2$, and $e_3$ of $\Z^3/\Z^3 M$ give rise to the generators $\pm 2$, $3$, and $-12=12$ $(\mod 24)$ of $\Z_{24}$; see the last column of $V$.

\subsection{Expansion and contraction of Abelian Cayley graphs}

The following basic results are simple consequences of the close relationship between the Cartesian product of Abelian Cayley graphs
and the direct products of Abelian groups (see, for instance, \cite{f87,f95}).

\begin{lemma}
\label{basic-lemma}
\begin{itemize}
\item[$(i)$]
The Cartesian product  of the Abelian Cayley graphs $G_1=\Cay(\G_1,\Sigma_1)$ and $G_2=\Cay(\G_2,\Sigma_2)$
is the Abelian Cayley graph $G_1\times G_2=\Cay(\G_1\times \G_2,(\Sigma_1,0)\cup (0,\Sigma_2))$.
In terms of congruences, if $\G_1=\Z^{n_1}/\Z^{n_1}M_1$,  $\G_2=\Z^{n_2}/\Z^{n_2}M_2$, $\Sigma_1=\{e_1,\ldots,e_{n_1}\}$, and
$\Sigma_2=\{e_1,\ldots,e_{n_2}\}$, then
$G_1\times G_2=\Cay(\Z^{n_1+n_2}/\Z^{n_1+n_2}M, \Sigma)$, where $M$ is the block-diagonal matrix $\diag(M_1,M_2)$ and $\Sigma=\{e_1,\ldots,e_{n_1+n_2}\}$.
\item[$(ii)$]
Let us consider the Cayley Abelian graph $G=\Cay(\G,\{a_1,\ldots,a_n,b\})$ with diameter $D$, where $b$ is an involution. Then, the quotient graph $G'=G/K_2$, obtained from $G$ by contracting all the edges generated by $b$, is an Abelian Cayley graph on the quotient group $\G/\Z_2$, with $n$ generators and diameter $D'\in\{D-1,D\}$.
\item[$(iii)$]
For a given integer matrix $M$ with a row $u$, let $G=\Cay(\Z^n/\Z^nM,\{e_1,\ldots,e_n\})$ have diameter $D$. Then, for some integer $\alpha>1$, the graph $G'=\Cay(\Z^n/\Z^nM',\{e_1,\ldots,e_n,$ $2u,\ldots,\alpha u\})$, where $M'$ is obtained from $M$ multiplying $u$ by $\alpha$, has diameter $D'=D+1$.
\end{itemize}
\end{lemma}

\section{A new approach to the Moore bound for mixed Abelian Cayley graphs}\label{sec:new-approach}

In \cite{lpp17}, L\'opez, P\'erez-Ros\'es, and Pujol\`as  derived the Moore bound \eqref{Moore-mix-bip} by using recurrences and generating functions. In this section, we obtain another expression for $M_{AC}(r_1,r_2,z,k)$ in a more direct way from combinatorial reasoning.
In this context, recall that the number of ways of placing $n$ (undistinguished) balls in $m$ boxes is the combinations with repetition $CR_{m,n}={m+n-1\choose n}$.

\begin{proposition}
\label{newbound1}
Let $\G$ be an Abelian group with generating set $\Sigma$ containing $r_1(\le 1)$ involutions, $r_2$ pairs of generators $(a,-a)$, and $z$ generators $b$ with $-b\not\in \Sigma$. Then, the number of vertices of the Cayley graph $G=\Cay(\G,\Sigma)$ is bounded above by
\begin{align}
M_{AC}(r_1,r_2,z,k)& =\sum_{i=0}^{r_2}{r_2\choose i}2^i\sum_{j=0}^{r_1}{r_1\choose j}{k+z-j\choose i+z}.\label{newbound}
\end{align}
\end{proposition}

\begin{proof}
A vertex $u$ at distance at most $k$ from $0$ can be represented by the situation of $k$ balls (representing the presence/absence of the edges/arcs in the shortest path from $0$ to $u$) placed in $1+r_1+r_2+z$ boxes (representing the presence/absence of the generators) with the following conditions:
\begin{itemize}
\item
One box contains the number of (white) balls of the non-existing edges/arcs.
\item
Each of the $r_1$ boxes contains at most one (white) ball of the edge defined by the corresponding involution.
\item
Each of the $r_2$ boxes contains a number of balls, which are either all white or all black, of the edges defined by the corresponding generator $a$ (white) or $-a$ (black).
\item
Each of the $z$ boxes contains a number of (white) balls of the arcs defined by the corresponding generator $b$ (with $-b\not\in \Sigma$).
\end{itemize}
Suppose that exactly $i$ of the $r_2$ boxes and $j$ of the $r_1$ boxes are non-empty. This gives a total of ${r_2\choose i}2^i{r_1\choose j}$ possibilities
(the term $2^i$ accounts for the two possible colors of all balls in each of the $r_2$ boxes). Then, there are $k-i-j$ balls left to be placed in $1+i+z$ boxes, which gives a total of ${k+z-j\choose i+z}$ situations. Joining all together, we obtain the result.
\end{proof}
In particular, \eqref{newbound} yields the known Moore bounds for the Abelian Cayley digraphs $(r_1=r_2=0)$, and Abelian Cayley graphs with no involutions $(r_1=z=0)$. Namely,
$$
M_{AC}(0,0,z,k)={k+z\choose z}\qquad\mbox{and}\qquad
M_{AC}(0,r_2,0,k)=\sum_{i=0}^{r_2} 2^i{r_2\choose i}{k\choose i},
$$
respectively. See Wong and Coppersmith \cite{Wong74}, and Stanton and Cowan \cite{sc70}, respectively.
\begin{table}[htb]
\centering
\begin{tabular}{|c||c|c|c|c|c|}
\hline
$z\backslash r_1$ & $0$ & $1$ & $2$ & \ldots & $r_1$  \\
\hline
\hline
$0$ & {\boldmath $1$} & $2$ & $4$ & \ldots & {\boldmath $2^{r_1}$}\\
\hline
$1$ & $k+1$ & $2k+1$ & $4k$ & \ldots  & \vdots \\
\hline
$2$ & ${k+2\choose 2}$ & $(k+1)^2$ & $2k^2+2k+1$ & \ldots   & \vdots \\
\hline
$3$ & ${k+3\choose 3}$ & ${k+3\choose 3}+{k+2\choose 3}$ & ${k+3\choose 3}+2{k+2\choose 3}+{k+1\choose 3}$ & \ldots   &  \vdots \\
\hline
\vdots & \vdots &  \vdots &  \vdots & $\ddots$  &  \vdots \\
\hline
$z$ & {\boldmath ${k+z\choose z}$} & ${k+z\choose z}+{k+z-1\choose z}$  &  \dots &  \dots &  {\boldmath $\sum_{j=0}^{r_1}{r_1\choose j}{k+z-j\choose z}$}\\
\hline
\end{tabular}\caption{Some values of the Moore bound for the mixed Abelian Cayley graphs with $r_2=0$ and diameter $k\geq 2$. The extreme cases are in boldface: $z=0$, which corresponds to the hypercubes; $r_1=0$, which are the Cayley digraphs; and the case for general values of $r_1$ and $z$. }
\label{tab:z2=0}
\end{table}

\subsection{Algebraic symmetries}
As commented in the Introduction, the numbers $M_{AC}(r_1,r_2,z,k)$ have some symmetries. For instance, as it is well-known, the so-called {\it Delannoy} numbers $F_{t,k}$, which correspond to $M_{AC}(0,t,0,k)$, satisfy
$$
F_{t,k}=\sum_{i=0}^t 2^i{t\choose i}{k\choose i}=\sum_{i=0}^k 2^i{k\choose i}{t\choose i}=F_{k,t},
$$
where we have used \eqref{newbound}.
Also, we have already mentioned that $$M_{AC}(0,1,1,k)=M_{AC}(1,0,2,k)=(k+1)^2.$$
In fact, this is a particular case of the following result.
\begin{lemma}
For any integer $\nu$ such that $-r_2\le \nu\le \min\{r_1,z\}$,  the Moore bounds for the mixed Abelian Cayley graphs satisfy
  \begin{equation}
  \label{symmetry}
  M_{AC}(r_1,r_2,z,k)=M_{AC}(r_1-\nu,r_2+\nu,z-\nu,k).
  \end{equation}
\end{lemma}
\begin{proof}
We only need to prove it for $\nu=1$. Remembering the proof of Proposition \ref{newbound1}, notice that there is an equivalence between $(i)$ and $(ii)$ as follows:\\
$(i)$ The balls of a box corresponding to a generator $b\in \Sigma$ (with $-b\not\in \Sigma$) together with the (0 or 1) balls of the box representing an involution $\iota$; \\
$(ii)$ The (`white' or `black') balls in a box representing the pair of generators $\pm b\in \Sigma$. So, in our counting process, each generator pair of type $\{\iota,b\}$
can be replaced by a generator of type $a$, without changing the result.

In fact, notice that a more direct proof is obtained by using the expression \eqref{eq:upper1} for $M_{AC}(r_1,r_2,z,k)$, since it is invariant under the changes $r_1\rightarrow r_1-\nu$,  $r_2\rightarrow r_2+\nu$, and  $z\rightarrow z-\nu$.
\end{proof}
By combining \eqref{symmetry} for the extreme values of $\nu$ with \eqref{newbound}, we get the following result.

\begin{theorem}
Depending on the values of $r_1$, $r_2$, and $z$, the Moore bounds for the mixed Abelian Cayley graphs are:
\begin{itemize}
\item[$(i)$]
For any values of $r_1,r_2,z$,
\begin{equation}
\label{eq1-theo}
M_{AC}(r_1,r_2,z,k)=\sum_{j=0}^{r_1+r_2}{r_1+r_2\choose j}{k+r_2+z-j\choose z+r_2}.
\end{equation}
\item[$(ii)$]
If  $r_1\le z$,
\begin{equation}
M_{AC}(r_1,r_2,z,k)=\sum_{i=0}^{r_1+r_2}{r_1+r_2\choose i}{k+z\choose i+z}2^i.
\end{equation}
\item[$(iii)$]
If  $r_1\ge z$,
\begin{equation}
M_{AC}(r_1,r_2,z,k)=\sum_{i=0}^{r_2+z}{r_2+z\choose i}2^i\sum_{j=0}^{r_1-z}{r_1\choose j}{k-j\choose i}.
\end{equation}
\end{itemize}
\end{theorem}
\begin{proof}
$(i)$ By taking $\nu=-r_2$ in \eqref{symmetry}, we get $M_{AC}(r_1,r_2,z,k)=
M_{AC}(r_1+r_2,0,z+r_2,k)$, and the result follows by applying \eqref{newbound}.
The results in $(ii)$ and $(iii)$ are proved analogously by taking $\nu=r_1$ and $\nu=z$, respectively, in \eqref{symmetry}.
\end{proof}
Notice that, by changing the summation variable $j$ to $k-i$ in
\eqref{eq1-theo}, we get \eqref{eq:upper1}, so proving that the expressions given for $M_{AC}(r_1,r_2,z,k)$ in \eqref{eq:upper1} and
\eqref{newbound} are equivalent.

\section{Existence of mixed Moore Abelian Cayley graphs and some dense families}
\label{sec:requalto1}

Graphs (or digraphs or mixed graphs) attaining the Moore bound are called {\em Moore graphs} (or {\em Moore digraphs} or {\em mixed Moore graphs}). The same applies for particular versions of the problem and, hence, mixed Abelian Cayley graphs attaining the Moore bound are called {\em mixed Moore Abelian Cayley graphs}. Moore graphs are very rare and, in general, they only exist for a few values of the degree and/or diameter.
In this section, we deal with the existence problem of these extremal graphs and, if they do not exist, we give some infinite families of dense graphs.

\subsection{The one-involution case}
We begin with the case when $r=1$, that is, when our graphs contain a $1$-factor. This means that the corresponding generating set of the group must contain exactly one involution. If, in addition, we have just one arc generator, that is $z=1$, then the Moore bound \eqref{eq:upper1} becomes $2k+1$. This bound is unattainable since a graph containing a $1$-factor must have even order. Nevertheless, it is easy to characterize those mixed Abelian Cayley graphs with maximum order $2k$ in this case: Let $\Gamma$ be an Abelian group and $\Sigma=\{\iota,b\}$ a generating set of $G$ , where $\iota$ is an involution of $\Gamma$ (the edge generator) and $b$ is an element whose inverse is not in $\Sigma$ (the arc generator). The set of vertices at distance $l$ from $0$ is $G_l(e)=\{\iota+(l-1)b,lb\}$, for $1 \leq l \leq k-1$. Since the order of $G$ is $2k$ and the diameter is $k$, then $G_i(e) \cap G_j(e) = \emptyset$ for $i \neq j$, and moreover, $G_k(e)=\{\iota+(k-1)b\}$. Now, due to the regularity of $G$, we have two possibilities according to the arc emanating from $(k-1)b$, that is, the value of $kb$:
\begin{itemize}
 \item[(a)]
 $kb=\iota$. Then, $\iota+(k-1)b=2\iota=0$, and the mixed graph has been completed. In this case, $b$ has order $2k$, that is, $b$ is a generator of $\Gamma$. Thus, $G$ contains a Hamiltonian directed cycle (generated by $b$), and it is trivial to see that $G$ is isomorphic to a circulant graph of order $2k$ with generators $1$ and $k$.
 \item[(b)]
 $kb=0$. Then, the arc emanating from $\iota (k-1)b$ goes to $\iota$, and the mixed graph has been completed. This graph contains two disjoint directed cycles of order $k$, both joined by a matching. Then, $G$ is isomorphic to $\Cay(\mathbb{Z}_2 \times \mathbb{Z}_k;\{(1,0),(0,1)\})$.
\end{itemize}

So, a mixed Abelian Cayley graph with $r=z=1$ and maximum order for diameter $k$ is isomorphic either to $\Circ(2k;\{1,k\})$ or $\Cay(\mathbb{Z}_2 \times \mathbb{Z}_k;\{(1,0),(0,1)\})$.

For $z>1$, the problem seems to be much more difficult, but we have a complete answer for the case $z=2$.

\begin{theorem}
	\label{propo-r=1}
Depending on the value of the diameter $k\ge 2$, the maximum order of a mixed Abelian Cayley graph with $r_1=1$, $r_2=0$, and $z=2$ is given in Table \ref{tab:optimal}, together with some graphs attaining the bound.
\end{theorem}
\begin{table}[h!]
	\centering
	\small
	\begin{tabular}{|c|c|c|c|c|}
		\hline
		$k$ & $M(1,0,2,k)$ & $N$ & $\Gamma$ & $\Sigma$
		 \\
		\hline\hline
		2 & 9 & 8 & $\mathbb{Z}_{8}$ & $\{1,3,4\}$
		\\
		\hline
		3 & 16 & 12 & $\mathbb{Z}_{12}$ & $\{1,4,6\}$
		\\
		\hline
		4 & 25 & 18 & $\mathbb{Z}_{18}$ & $\{1,4,9\}$
		\\
		\hline
		$3x-1$ & $9x^2$ & $6x^2$ &  $\mathbb{Z}_2\times \mathbb{Z}_{x}\times \mathbb{Z}_{3x}$ & $\{(1,0,0),(1,1,1),(1,3,2)\}$
		\\
	    ($x$ even) &  &  &  &
	    \\
		\hline
		$3x-1$ & $9x^2$ & $6x^2$ & $\mathbb{Z}_x\times \mathbb{Z}_{6x}$ &   $\{(0,3x), (1,-2), (3-x,3x-7) \}$
		\\
		($x$ odd) &  &  &  &
		\\
		\hline
		$3x-1$ & $9x^2$ & $6x^2$ &  $ \mathbb{Z}_{x}\times \mathbb{Z}_{6x}$ & $\{(0,3x),(3,2),(1,1)\}$
		\\
		\hline
		$3x$ & $9x^2+6x+1$  & $6x^2+4x$ & $\mathbb{Z}_{2}\times\mathbb{Z}_{N/2}$ & $\{(1,0),(1,1),(-3x,-3x)\}$
		\\
		($x$ even) &   &  &  &
		\\
		\hline
		$3x$ & $9x^2+6x+1$  & $6x^2+4x$ & $\mathbb{Z}_{N}$ & $\{N/2,N/2+1,3x^2-x\}$
		\\
		($x$ odd) &   &  &  &
		\\
		\hline
		$3x$ & $9x^2+6x+1$  & $6x^2+4x$ & $\mathbb{Z}_{N}$ & $\{N/2,2x+1,x\}$
		\\
		\hline
		$3x+1$ & $9x^2+12x+4$ & $6x^2+8x+2$ & $\mathbb{Z}_{N}$ & $\{N/2,N/2+1,6x^2+5x\}$
		\\
		($x$ even) &  &  &  & $\{N/2,N/2-3x-2,N/2+1\}$
		\\
		\hline
		$3x+1$ & $9x^2+12x+4$ & $6x^2+8x+2$ & $\mathbb{Z}_2\times\mathbb{Z}_{N/2}$ & $\{(1,0),(1,1),(-3x-2,-3x-2)\}$
		\\
		($x$ odd) &  &  & $\Z_N$ & $\{N/2,N/2-3x-2,1\}$
		\\
		\hline
	\end{tabular}\caption{Some mixed Abelian Cayley graphs with $r=1$, $r_2=0$, and $z=2$, with maximum order given the diameter $k\geq 2$ ($x>1$).}\label{tab:optimal}
\end{table}
\begin{proof}
Let $G$ be an Abelian Cayley graph with the above parameters and maximum order $N\le (k+1)^2$. 	Then, by Lemma \ref{basic-lemma}$(ii)$, the graph $G'=G/K_2$ is an Abelian  Cayley digraph with $z=2$ generators and diameter $k'\in\{k-1,k\}$.

For the values $k\le 4$, Table \ref{tab:optimal} shows the optimal values of $N$ found by computer search.

Otherwise, if $k>4$, we prove that $G'$ has diameter $k'=k-1$.
Indeed, note first that, since $G'$ is a circulant digraph with two `steps' (generators), it admits a representation as an $L$-shaped tile that tessellates the plane (see Fiol, Yebra, Alegre, and Valero \cite{fyav87}). Moreover, if such a tile $L$ has dimensions $\ell,h,x,y$ as shown in Figure \ref{fig1}, then $G'\cong \Cay(\Z^2/\Z^2M; \{e_1,e_2\})$ where $M'=\left(
\begin{array}{cc}
\ell & -y \\
-x & h
\end{array}
\right)$, and the diameter of $G'$ is $D'=D(L):=\max\{\ell+h+x-2,\l+k+y-2\}$ (the maximum of the distances from $0$ to the vertices `$\bullet$'). For instance, for the digraph $G=\Cay(\Z_{18;\{1,4,9\}})$ in Table \ref{tab:optimal}, the digraph $G'=G/K_2$ gives rise to the tessellation of Figure \ref{fig2} with
$M'=\left(
\begin{array}{cc}
4 & -1 \\
-3 & 3
\end{array}
\right)$ (notice that $\det M= 9$, as it should be).
In turn, by Lemma \ref{basic-lemma}$(iii)$, this implies that $G\cong \Cay(\Z^2/\Z^2M,\Sigma)$ with $M=\left(
\begin{array}{cc}
2\ell & -2y \\
-x & h
\end{array}
\right)$ and $\Sigma=\{(1,0),(0,1),(\ell,-y)\}$ or, equivalently,
$M=\left(
\begin{array}{cc}
2(\ell-x) & 2(h-y) \\
-x & h
\end{array}
\right)$ and $\Sigma=\{(1,0),(0,1),(\ell-x,h-y)\}$, which corresponds to take two equal $L's$, say $L^1$ and $L^2$, as shown in Figure \ref{fig2} (were $L^2$ is shaded). Now, since the diameter of $G$ is $k$, we have two possibilities:
\begin{itemize}
\item[$(i)$]
	The set of the two equal $L$'s, $L^1=L^2$, corresponds to a diagram of minimum distances (from $0$) with $D(L^1)=D(L^2)=k-1$.
	Notice that, in this case, $L^1$ and $L^2$ correspond to the sets of vertices, $G^1(0)$ and
	$G^2(0)$, at minimum  distance from $0$, whose respective shortest paths do, or do not, contain vertex $\iota$ (the involution), and we have $|G^1(0)|=|G^2(0)|$.
\item[$(ii)$]
	Otherwise, as happens in our example, the diagram of minimum distances is a single $L$ with $D(L)=k$, as shown in Figure \ref{fig4} $(c)$. This is because, apart from the case $(i)$, the only tile that tessellates the plane is the above $L^1$ ``extended" to $L$ by using of $L^2$. In this case, $|G^1(0)|\ge |G^2(0)|$.
\end{itemize}

\begin{figure}[t]
	\begin{center}
		\includegraphics[width=16cm]{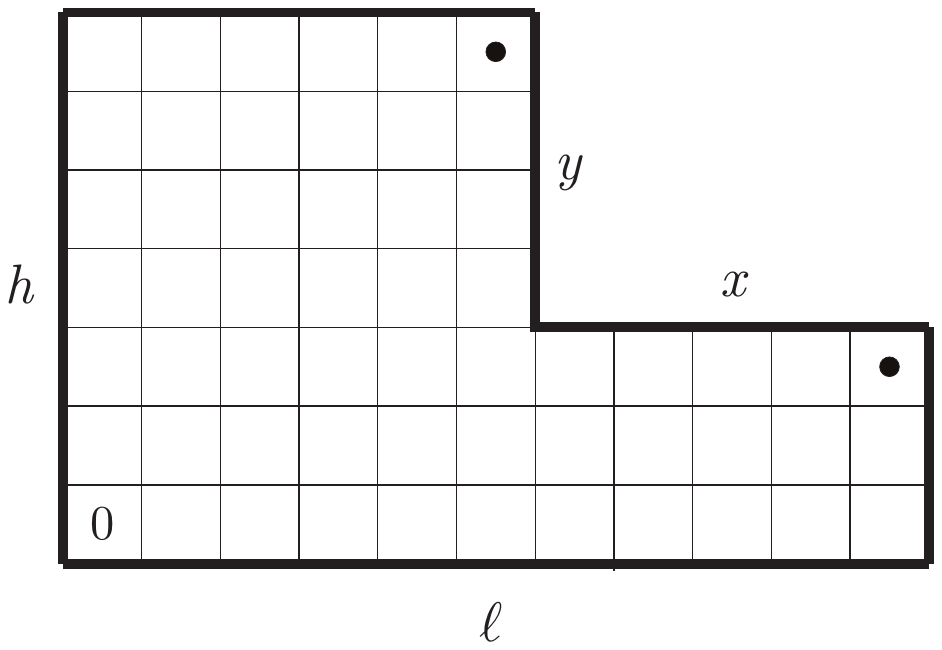}
	\end{center}
	\vskip -18cm
	\caption{An $L$-shaped tile.}
	\label{fig1}
\end{figure}

\begin{figure}[t]
	\begin{center}
		\includegraphics[width=20cm]{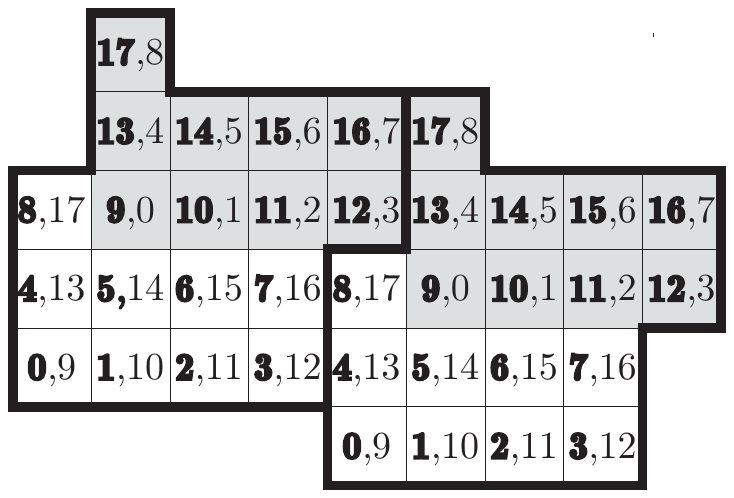}
		\end{center}
	\vskip-23.5cm
\caption{A plane tessellation of $\Cay(\Z_{18};\{1,4,9\})/K_2$.}
	\label{fig2}
\end{figure}

In \cite{fyav87} it was proved (by considering a continuous version of the problem), that the order $N'$ of a circulant digraph with two generators satisfies the bound  $N'\le\frac{(k'+2)^2}{3}$. Consequently, in the case $(i)$ and taking $k'=k-1$, the maximum order for $G=\Cay(\G,\{\iota,e_1,e_2\})$ turns out to be
\begin{equation}
\label{bound(i)}
N_{(i)}=2N'= \left\lfloor\frac{2}{3}(k+1)^2\right\rfloor.
\end{equation}
For the case $(ii)$, a similar reasoning with the (symmetric) tile $L=L^1+L^2$ shown in Figure \ref{fig3}, with  involution $\iota=(c,c)$ for $c>0$, area $N=\ell^2-(\ell-2c)^2$, and diameter
$$
k=\max\{D(L^1),D(L^2)+1\}=\max\{\ell+c-2, \ell-1\}=\ell+c-2,
$$
gives the maximum
\begin{equation}
\label{bound(ii)}
N_{(ii)}= \left\lfloor\frac{1}{2}(k+2)^2\right\rfloor.
\end{equation}
Then, from \eqref{bound(i)} and \eqref{bound(ii)}, we see that $N_{(ii)}> N_{(i)}$ only for $k\le 4$. Hence, we proved that,
if $k>4$, we can assume that $G'=G/K_2$, with $G$ having maximum order, has diameter $k'=k-1$, as claimed.

Now, using again the results in \cite{fyav87}, the maximum number of vertices of $G'$ are reached by the following graphs:
\begin{itemize}
\item
If $k'=3x-2$, then $N=3x^2$ attained by $\Cay(\Z^2/\Z^2M,\{e_1,e_2,e_3\})$\\ with
$
M=\left(
\begin{array}{cc}
2x & -x\\
-x & 2x
\end{array}\right).
$
\item
If $k'=3x-1$, then $N=3x^2+2x$ attained by
$\Cay(\Z^2/\Z^2M,\{e_1,e_2,e_3\})$\\ with
$
M=\left(
\begin{array}{cc}
2x & -x\\
-x & 2x+1
\end{array}\right).
$
\item
If $k'=3x$, then $N=3x^2+4x+1$ attained by
$\Cay(\Z^2/\Z^2M,\{e_1,e_2,e_3\})$\\ with
$
M=\left(
\begin{array}{cc}
2x+1 & -x\\
-x & 2x+1
\end{array}\right).
$		
\end{itemize}
Hence, by Lemma \ref{basic-lemma}, we have the following mixed Abelian Cayley graphs with $r=1$ and $z=2$ with maximum order $N=2N'$ for every diameter $k=k'+1>4$:
\begin{itemize}
	\item[$(i)$]
	If $k=3x-1$, then $N=6x^2$ attained by\\ $\Cay(\Z^3/\Z^3M,\{e_1,e_2,e_3\})$ with
	$
	M=\left(
	\begin{array}{ccc}
	2 & 0 & 0\\
	0 & 2x & -x\\
	0 & -x & 2x
	\end{array}\right),
	$\\
	or
	$\Cay(\Z^2/\Z^2M,\{(2x,-x),e_1,e_2\})$ with
	$
	M=\left(
	\begin{array}{cc}
	4x & -2x\\
	-x & 2x
	\end{array}\right).
	$
	\item [$(ii)$]
	If $k=3x$, then $N=6x^2+4x$ attained by\\
	$\Cay(\Z^3/\Z^3M,\{e_1,e_2,e_3\})$ with
	$
	M=\left(
	\begin{array}{ccc}
	2 & 0 & 0 \\
	0 & 2x & -x\\
	0 & -x & 2x+1
	\end{array}\right),\\
	$
	or
	$\Cay(\Z^2/\Z^2M,\{(2x,-x),e_1,e_2\})$ with
	$
	M=\left(
	\begin{array}{cc}
	4x & -2x\\
	-x & 2x+1
	\end{array}\right).
	$
	\item [$(iii)$]
	If $k=3x+1$, then $N=6x^2+8x+2$ attained by\\
	$\Cay(\Z^3/\Z^3M,\{e_1,e_2,e_3\})$ with
	$
	M=\left(
	\begin{array}{ccc}
	 2 & 0 & 0 \\
	0 & 2x+1 & -x\\
	0 & -x & 2x+1
	\end{array}\right),
	$
	\\	
	or
	$\Cay(\Z^2/\Z^2M,\{(2x+1,-x),e_1,e_2\})$ with
	$
	M=\left(
	\begin{array}{cc}
	4x+2 & -2x\\
	-x & 2x+1
	\end{array}\right).
$		
\end{itemize}

\begin{figure}[t]
	\begin{center}
		\includegraphics[width=15cm]{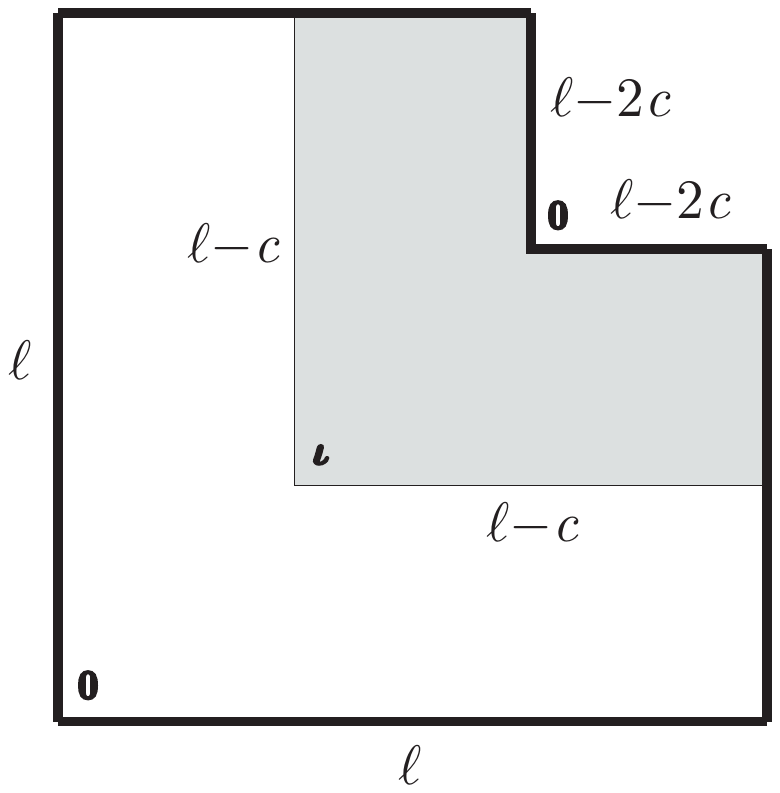}
	\end{center}
	\vskip-15.5cm
	\caption{A generic form of the tile $L=L^1+L^2$ with the involution $\iota=(c,c)$.}
	\label{fig3}
\end{figure}

Assume that $M$ is a $3\times 3$ matrix (the cases with $2\times 2$ matrices are more simple).
In the case $(i)$, we distinguish two cases: If $x$ is even, say $x=2s$, then
the Smith normal form of $M$ turns out to be $S=\diag(2,x,3x)=\diag(2,2s,6s)$ and
$$
S=UMV=\left(
\begin{array}{ccc}
1 & 0 & 0\\
-s & -1 & 0\\
0 & 5 & 1
\end{array}
\right)\left(
\begin{array}{ccc}
2 & 0 & 0\\
0 & 4s & -2s\\
0 & -2s & 4s
\end{array}
\right)\left(
\begin{array}{ccc}
1 & 0 & 0\\
1 & 1 & 1\\
3 & 3 & 2
\end{array}
\right).
$$
Thus, $\Cay(\Z^3/\Z^3M,\{e_1,e_2,e_3\})\cong\Cay(\mathbb{Z}_2\times \mathbb{Z}_{x}\times \mathbb{Z}_{3x},\{(1,0,0),(1,1,1),(3,3,2)\})$, as shown in Table \ref{tab:optimal}. The reasoning of the odd case, $x=2s+1$, is similar, with the Smith normal form being now $S=\diag(1,x,6x)$. Then, in this case, the group $\Z^3/\Z^3M$ is of rank two, namely, $\Z_x\times\Z_{6x}$.
The cases $(ii)$ and $(iii)$ are proved analogously. For instance, when $k=3x$ is even, the Smith normal form of $M$ is $S=\diag(1,2,N/2)$
(group of rank two); whereas when $k=3x$ is odd, we get $S=\diag(1,1,N)$ (cyclic group).
\end{proof}

Notice that, using our method, circulant graphs attain the maximum order in some cases, but not in all of them.
More precisely, for a diameter of the form $k=3x-1$, we must always use a group with rank $2$ or $3$, whereas for the other cases, $k\in\{3x,3x+1\}$, we can use a cyclic group.
Moreover, we remark that in some cases there are other generators and/or groups that produce non-isomorphic mixed graphs with the same degree, diameter and order than the ones given in Table \ref{tab:optimal}. For instance, $\mathbb{Z}_{2}\times\mathbb{Z}_{16}$ (with any of the three following sets $\{(1,11),(0,1),(0,8)\},\{(1,11),(1,0),(1,8)\}$ or $\{(1,11),(0,5),(1,8)\}$ as generators) produces an optimal mixed graph for $k=6$. Another example is given in Figure \ref{fig6}, where we show an alternative mixed graph with diameter $3x=6$, maximum order $32$ and generators $1$, $10$, $16$ (different from the one corresponding to $5$, $2$, $16$, provided by Table \ref{tab:optimal}).

\begin{figure}[t]
	\begin{center}
		\includegraphics[width=14cm]{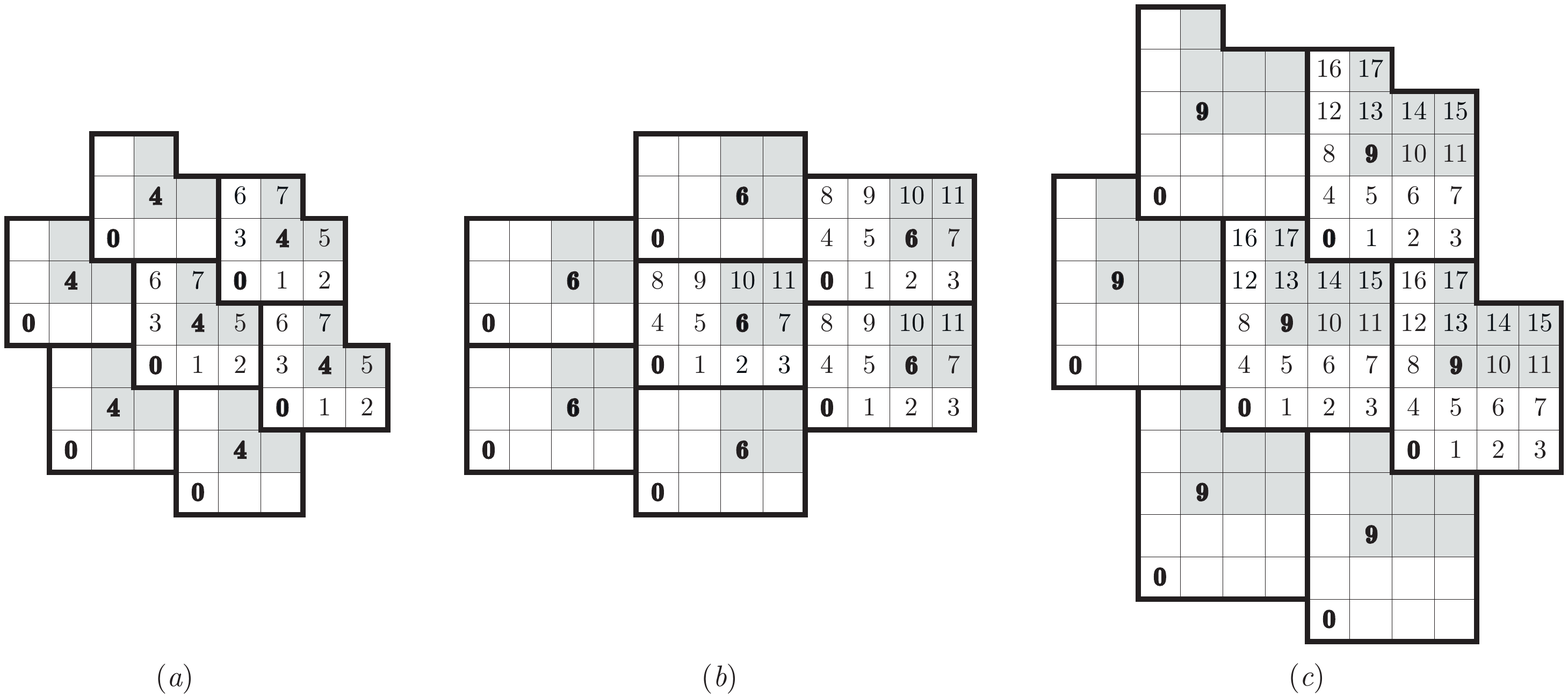}
	\end{center}
	\vskip-4cm
	\caption{The plane tessellations corresponding to the mixed Abelian Cayley graphs with $r_1=1$, $z=2$, diameters $(a)$ $k=2$, $(b)$ $k=3$, and $(c)$ $k=4$, and maximum numbers of vertices $8$, $12$, and $18$, respectively.}
	\label{fig4}
\end{figure}

\begin{figure}[h!]
	\begin{center}
		\includegraphics[width=10cm]{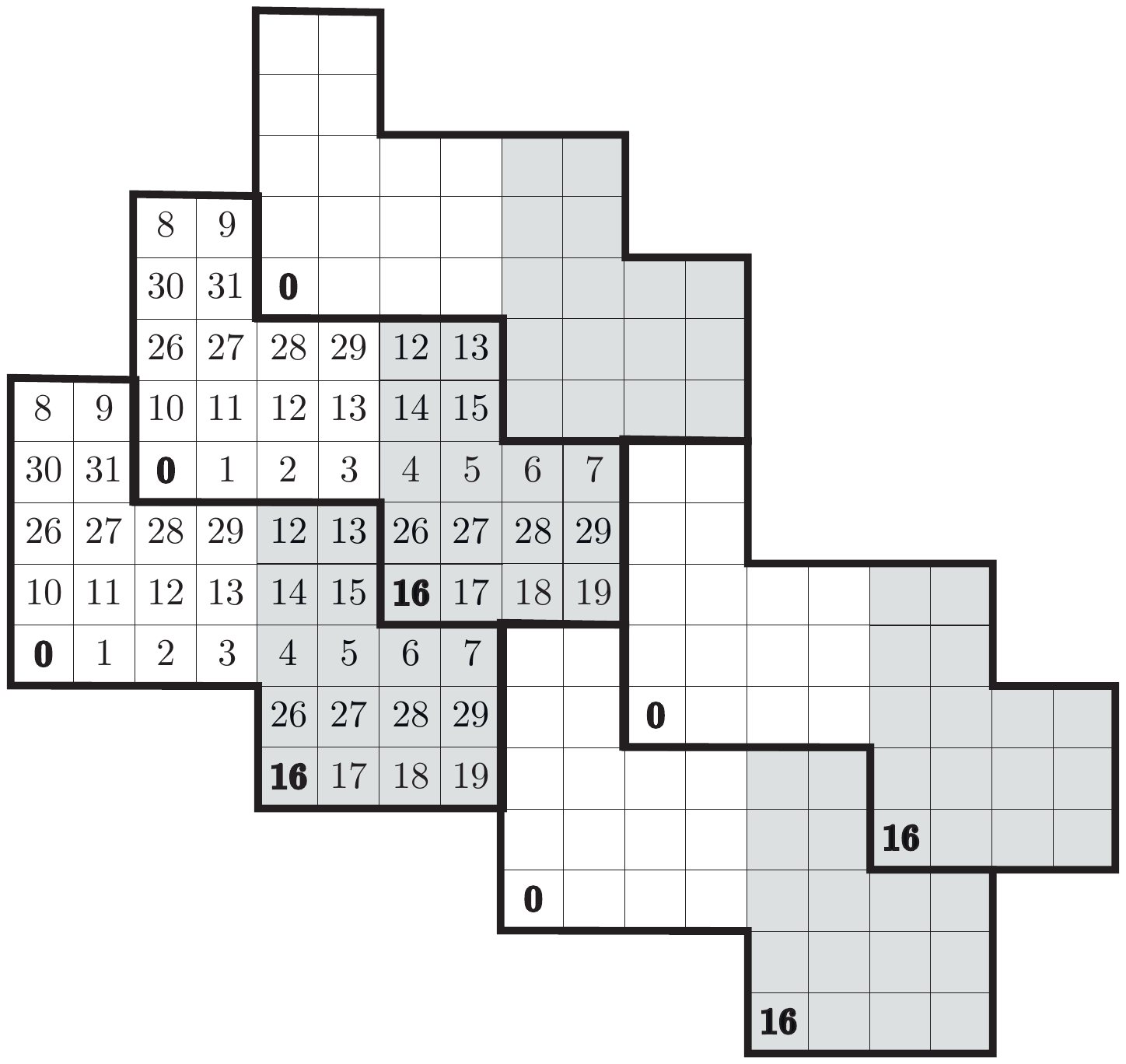}
	\end{center}
	\vskip-7.6cm
	\caption{The plane tessellation of the mixed Abelian Cayley graph $\Cay(\Z_{32},\{1,10,16\})$.}
	\label{fig5}
\end{figure}

\begin{figure}[t]
	\begin{center}
		\includegraphics[width=7cm]{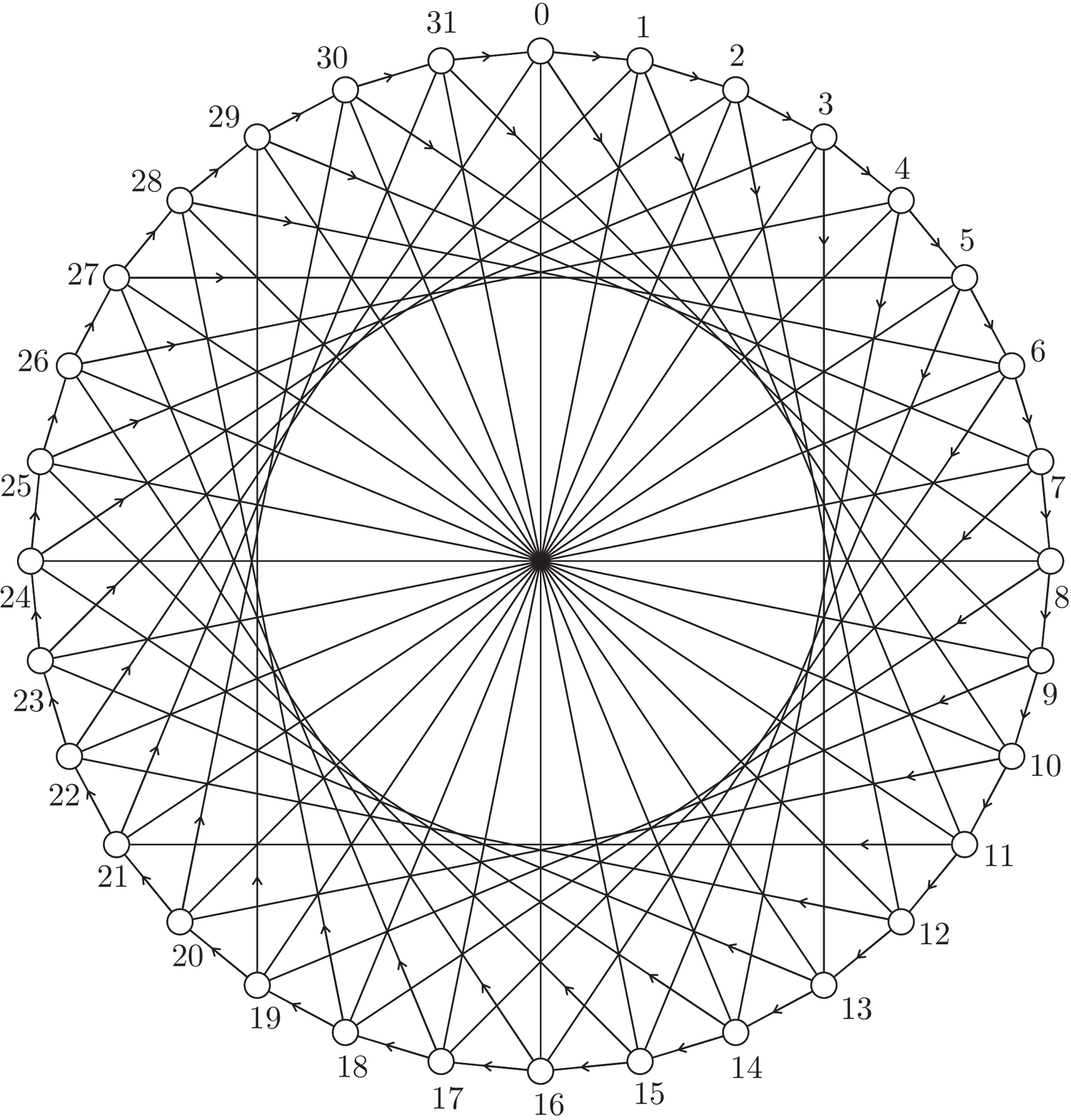}
	\end{center}
\vskip -3cm
	\caption{The mixed Abelian Cayley graph $\Cay(\Z_{32},\{1,10,16\})$ with diameter $k=6$ and maximum number of vertices.}
	\label{fig6}
\end{figure}

As happens with the case $z=1$, the Moore bound cannot be attained either for $z\geq 2$. Now we have the following result.

\begin{theorem}\label{theo:nonexist}
There are no mixed Moore Abelian Cayley graphs for $r=1$, $z\geq 2$ and $k \geq 2$.
\end{theorem}

\begin{proof}
Let $\Sigma=\{\iota,b_1,\dots,b_z\}$ be a set of generators of an Abelian group $\Gamma$ of order $M_{AC}(1,0,z,k)$, where $\iota$ is an involution. Let us consider the Abelian Cayley mixed graph $G=\Cay(\Gamma,\Sigma)$, and suppose that the diameter of $G$ is $k$. The set of vertices at distance $l$, $1 \leq l \leq k$ from $0$, $G_l(0)$, can be split into two disjoint sets:  $G_l^1(0)=\{\sum s_ib_i \ | \ \sum s_i=l \}$ and $G_l^2(0)=\{\iota + \sum s_ib_i \ | \ \sum s_i=l-1\}$. The shortest path from $b_1$ to $b_t$, for all $2 \leq t \leq z$, must pass through a vertex $v_t \in G_k(0)$.
\begin{itemize}
 \item[] \emph{Claim:} For any $2 \leq t \leq z$, there exists $v_t \in G_k(0)$, such that $b_1+v_t=b_t$. Indeed, let $v'_t \in G_k(0)$ be the predecessor of $b_t$ in the shortest path from $b_1$ to $b_t$. Then, there exists $b_i$, for $1 \leq i \leq z$, such that $v'_t+b_i=b_t$. We already know that $v'_t=\sum s_ib_i$, with $s_i \geq 0$ and $\sum s_i=k$, but since the shortest path starts at $b_1$, then we have the extra condition that $s_1 \geq 1$. Hence, $v'_t+b_i=b_1+\sum s'_ib_i$, where $s'_i \geq 0$ and $\sum s'_i=k$. That is, $b_t=v'_t+b_i=b_1+v_t$ for a vertex $v_t \in G_k(0)$.
\end{itemize}
Now, we have two possibilities:
\begin{itemize}
 \item[(a)] $v_t \in G_k^1(0)$. Then, $b_1+\sum s_ib_i = b_t$, for some vector $(s_1,\dots,s_z)$, where $0 \leq s_i \leq k$, and $\sum s_i=k$. That is, $w_t=(s_1+1)b_1+ \dots +(s_t-1)b_t+ \dots +s_zb_z =0$. Observe that $w_t \in G_k(0)$ if $s_1 < k$, which is a contradiction with $0 \in G_k(0)$ for $k \geq 1$. Hence, $s_1=k$, that is, $(k+1)b_1=b_t$. Equivalently, $v_t=kb_1$, which means that $v_t$ does not depend on the vertex $b_t$.
 \item[(b)]  $v \in G_k^2(0)$. As in the previous case, it is not difficult to see that $v_t=\iota+(k-1)b_1$, hence also in this case $v_t$ does not depend on the vertex $b_t$.
\end{itemize}
Hence, for $z \geq 4$, we have at least two different vertices $b$ and $b'$ from the set $\{b_2,\dots,b_z\}$, such that either $(k+1)b_1=b$ and $(k+1)b_1=b'$ or $\iota+(k-1)b_1=b$ and $\iota+(k-1)b_1=b'$, which is a contradiction. It remains to consider the cases $z=2$ and $z=3$. The first one can be solved as follows: Using the reasoning given in (a) and (b), we have that either $(k+1)b_1=b_2$ or $\iota+kb_1=b_2$. Now, apply the same argument to the shortest path from $b_1$ to $2b_2$($\neq b_2$ since $k \geq 2$), showing that either $(k+1)b_1=2b_2$ or $\iota+kb_1=2b_2$, which is a contradiction with the two cases given before. To solve the last case $z=3$, it is enough to use the same argument to the shortest path from $b_1$ to $2b_3$, in addition to the others described before.
\end{proof}


For $r=1$ and $z \geq 1$, the upper bounds \eqref{eq:upper1} and \eqref{newbound} become (see Table \ref{tab:z2=0})
\begin{equation}
\label{Moore(r=1,z)}
M_{AC}(1,0,z,k)=\sum_{i=0}^1 {1 \choose i}{z+k-i \choose k-i}=\frac{2k+z}{k+z}{k+z \choose k},
\end{equation}
which is asymptotically close to $\frac{2k^z}{z!}$ for large $k$ (see L\'opez, P\'erez-Ros\'es, and Pujol\`as \cite{lpp17}). In the same paper, it was shown that, for every $z\geq 1$ and every even $n>2$, the diameter of the mixed circulant graph $\Circ(n^z;\{1,n,n^2,\dots,n^{z-1},\frac{1}{2}n^z\})$ is $k=(z-1)(n-1)+\frac{n}{2}$. That is, if $\frac{2k-1}{2z-1}$ is an odd integer at least $3$, such a mixed graph has order
\begin{equation}
\label{constructionUdL}
N=\left(1+\frac{2k-1}{2z-1}\right)^z.
\end{equation}
Moreover, as a consequence, the mixed graph $\Circ(n^z;\{1,n,n^2,\dots,n^{z-1},\frac{1}{2}n^z\})$ approaches the upper bound asymptotically, since the diameter $k$ increases:
\begin{equation}
\label{limit1}
\lim_{k \rightarrow \infty}
\frac{(1+\frac{2k-1}{2z-1})^z}{\frac{2k+z}{k+z}{k+z \choose k}}
= \frac{2^{z-1}z!}{(2z-1)^z}.
\end{equation}
This means that for any value of the directed degree $z$, there is a construction that approaches the upper bound by a factor that is a function depending only on $z$. This approximation is good only for small values of $z$. For instance, $\Circ(n^2;\{1,n,\frac{1}{2}n^2\})$ approaches the upper bound by the factor $\frac{4}{9}$.

The following result shows a better family of dense mixed graphs with $r=1$ and $z\ge 2$.
\begin{proposition}
Let us consider the Abelian group $\Gamma=\Z_2\times\Z_m\times\Z_{m(z+1)}\times\stackrel{(z-1)}{\ldots\ldots}\times \Z_{m(z+1)}$, with generating set $\Sigma=\{(1,0,0,\ldots,0),(0,1,1,\ldots,1),(0,2,1,\ldots,1),\stackrel{}{\ldots\ldots},$ $(0,1,1,\ldots,2)\}$.
Then, the Cayley mixed graph $G=\Cay(\G,\Sigma)$, with $r=1$ and $z\ge 2$, has diameter $k$ and number of vertices
\begin{equation}
\label{N(z,k)}
N(z,k)=\frac{2^z}{z+1}\left(\frac{k-1}{z}+1 \right)^z.
\end{equation}
\end{proposition}
\begin{proof}
In Aguil\'o, Fiol, and P\'erez \cite[Th. 6]{afp16}, it was proved that the Cayley digraph $G'=\Cay(\Gamma',\Sigma')$ with $\Gamma'=\Z_m\times\Z_{m(n+1)}\times\stackrel{(n-1)}{\ldots\ldots}\times \Z_{m(n+1)}$ and $\Sigma'=\{(1,1,\ldots,1),(2,1,\ldots,1),$ $\ldots,(1,1,\ldots,2)\}$ has degree $d=n$ and diameter $k'={d+1\choose 2}m-d$. Thus, in terms of $d$ and $k'$, $G'$ has
$
N'=\frac{2^{d}}{d+1}\left(1+\frac{k'}{d}\right)^{d}
$
vertices. Now, the Cartesian product $G=K_2\times G'$ corresponds to the Cayley mixed graph described in the statement, with $r_1=1$, $z=d$, diameter $k=k'+1$, and order $N=2N'$. From these values, we get \eqref{N(z,k)}.
\end{proof}
This family also approaches the upper bound asymptotically for large values of the diameter. Namely,
\begin{equation}
\label{limit2}
\lim_{k \rightarrow \infty}
\frac{\frac{2^z}{z+1}\left(\frac{k-1}{z}+1\right)^z}
{\frac{2k+z}{k+z}{k+z \choose k}}
= \frac{2^{z-1}z!}{(z+1)z^z},
\end{equation}
improving the result in \eqref{limit1} when $z>2$. For example, for $z=3$, the limit in \eqref{limit1} is $24/125$, whereas the limit in \eqref{limit2} is $2/9$.
For finite values, the improvement is more noteworthy as $z$ and $k$ increase, as shown in Figure \ref{fig7}
for $z=5$ and $k\le 10$.

\begin{figure}[t]
\begin{center}
\includegraphics[width=8cm]{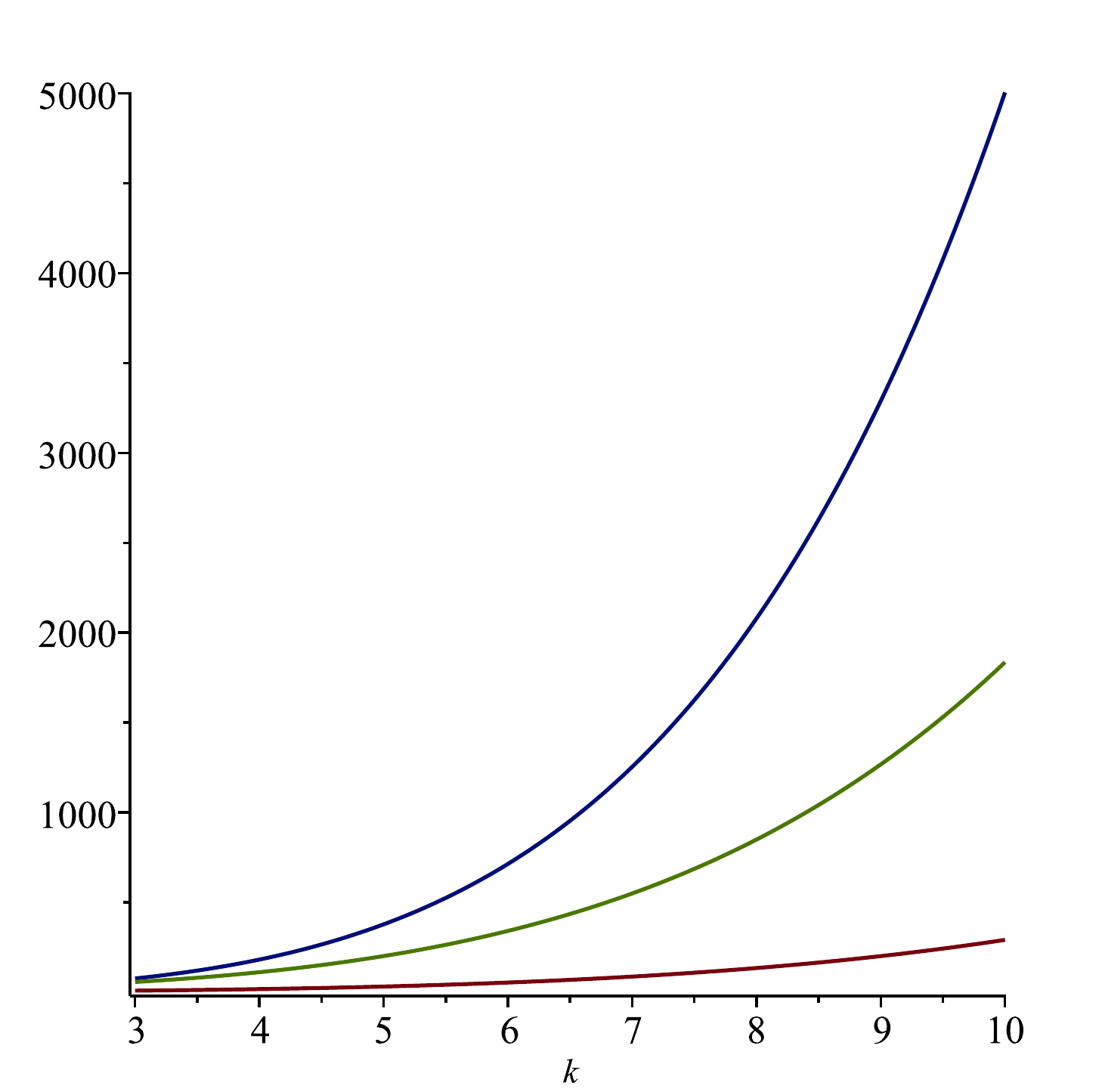}
\end{center}
\caption{Comparison, for $z=5$, between the Moore bound \eqref{Moore(r=1,z)} 
	(the uppermost function),
	and the numbers of vertices in  \eqref{N(z,k)} 
	(in the middle)
	 and \eqref{constructionUdL} 
	 (the lowest function).}
\label{fig7}
\end{figure}



\section*{Acknowledgments}
\label{sec:acknow}
The first two authors have been partially supported by the project 2017SGR1087 of the Agency for the Management of University and Research Grants (AGAUR) of the Catalan Government, and by MICINN from the Spanish Government under project PGC2018-095471-B-I00. The first and the third authors have been supported in part by grant MTM2017-86767-R of the Spanish Government.

\end{document}